\documentclass{amsart}
\usepackage{amsfonts}
\usepackage{amsmath}
\usepackage{amssymb}
\usepackage{enumerate}
\usepackage[utf8]{inputenc}
\usepackage{color}

\numberwithin{equation}{section}

\newtheorem{teo}{Theorem}[section]
\newtheorem{prop}[teo]{Proposition}
\newtheorem{lem}[teo]{Lemma}
\newtheorem{cor}[teo]{Corollary}

\theoremstyle{remark}

\newtheorem{example}[teo]{Example}

\title[Extended Koksma's inequality]{An extension of Koksma's inequality}
\author{Martin Lind}
\address{Department of Mathematics, Karlstad University, Universitetsgatan 2, 651 88 Karlstad, Sweden}
\email{martin.lind@kau.se}

\keywords{Discrepancy; Numerical integration; quasi-Monte Carlo method; Bounded variation; Convergence rates}
\subjclass[2010]{65D30, 26B30}

\begin{document}

\begin{abstract}
    When applying the quasi-Monte Carlo (QMC) method of numerical integration of univariate functions, Koksma's inequality provides a basic estimate of the error in terms of the discrepancy of the used evaluation points and the total variation of the integrated function. We present an extension of Koksma's inequality that is also applicable for functions with infinite total variation.
\end{abstract}

\maketitle

\section{Introduction}
%In this note we discuss some observations relating to error estimates for the \emph{quasi-Monte Carlo method} (QMC) of numerical integration on the interval $[0,1]$. 

A natural way to numerically calculate the integral of a function $f:[0,1]\rightarrow\mathbb{R}$ is to take a sequence ${\bf x}=\{x_n\}_{n=1}^\infty\subset[0,1]$ and use the approximation
\begin{equation}
    \label{MainApproximation}
    \int_0^1f(t)dt\approx\frac{1}{N}\sum_{n=1}^N f(x_n).
\end{equation}
Introduce the error 
\begin{equation}
    \label{MainError}
    \mathcal{E}_N(f;{\bf x})=\left|\frac{1}{N}\sum_{n=1}^Nf(x_n)-\int_0^1f(t)dt\right|.
\end{equation}
In the \emph{Monte Carlo method} (MC), one takes ${\bf x}$ in (\ref{MainApproximation}) to be a sequence of random numbers sampled uniformly from $[0,1]$. The expression inside the absolute value signs of (\ref{MainError}) is then a random variable with expected value 0 and standard deviation of the order $1/\sqrt{N}$ as $N\rightarrow\infty$, see e.g. \cite{Caflisch}.

The \emph{quasi-Monte Carlo method} (QMC) is based on instead taking a deterministic ${\bf x}\subset[0,1]$ in (\ref{MainApproximation}) with good properties (more precisely described below). This can lead to a better convergence rate of (\ref{MainError}) than when taking random {\bf x}. In fact, there exist deterministic ${\bf x}\subseteq [0,1]$ such that the rate of decay of $\mathcal{E}_N(f;{\bf x})$ is close to $1/N$ as $N\rightarrow\infty$ (see below). 
Useful references for the theory and applications of the MC and QMC methods are e.g. \cite{Caflisch, DP, Owen}.

The aim of this note is to discuss error estimates for QMC on $[0,1]$. More specifically, we establish an extension of the elegant \emph{Koksma’s inequality}. Koksma’s inequality is the main general error estimate for QMC, we need some auxiliary notions in order to formulate it.

Let ${\bf x}=\{x_n\}_{n=1}^\infty\subseteq[0,1]$ and let $E\subset[0,1]$. Denote for $N\in\mathbb{N}$
$$
A_N(E,{\bf x})=\sum_{n=1}^N\chi_E(x_n),
$$
where $\chi_E(x)=1$ if $x\in E$ and $\chi_E(x)=0$ if $x\notin E$.
Note that $A_N$ counts how many of the first $N$ terms of ${\bf x}$ that belong to $E$. The \emph{extreme discrepancy} of ${\bf x}$ is defined as 
\begin{equation}
    \nonumber
    D_N({\bf x})=\sup_{0\le a<b\le1}\left|\frac{A_N([a,b),{\bf x})}{N}-(b-a)\right|.
\end{equation}
If we set $a=0$ and take supremum over all $b\in(0,1]$, then we obtain the \emph{star discrepancy} of ${\bf x}$, denoted $D_N^*({\bf x})$. It is clear that
\begin{equation}
\label{discComparison}
D_N^*({\bf x})\le D_N({\bf x})\le 2D_N^*({\bf x}) 
\end{equation} 
In a sense, the quantities $D_N, D_N^*$ measure how much the distribution of the points of ${\bf x}$ deviates from the uniform distribution.

The \emph{total $p$-variation} of a function $f:[0,1]\rightarrow\mathbb{R}$ is given by
\begin{equation}
    \nonumber
    {\rm Var}_p(f)=\sup\left(\sum_{j=1}^\infty|f(I_j)|^p\right)^{1/p},
\end{equation}
where the supremum is taken over all non-overlapping collections of intervals $\{I_j\}_{j=1}^\infty$ contained in $[0,1]$ and $f(I)=f(b)-f(a)$ if $I=[a,b]$. If ${\rm Var}_p(f)<\infty$ we say that $f$ has \emph{bounded $p$-variation} (written $f\in BV_p$). Note that 
$$
BV_1\hookrightarrow BV_p\quad(1<p<\infty).
$$
See \cite{Dudley} for a thorough discussion of bounded $p$-variation and applications.

Koksma's inequality states that
\begin{equation}
    \label{Koksma}
    \mathcal{E}_N(f;{\bf x})\le D_N^*({\bf x}){\rm Var}_1(f).
\end{equation}
In other words, the error of QMC is bounded by a product of two factors, the first
measuring the ”spread” of the sequence ${\bf x}$ and the second measuring the variation of the integrand $f$.
We mention that there is a sequence ${\bf x}_C$ called the \emph{van der Corput sequence} (see \cite{KN}) such that $D_N^*({\bf x}_C)=\mathcal{O}(\log(N)/N)$, and this is the best rate of decay one can expect on $D_N^*$ (see Section 2 below). Hence, if $f\in BV_1$ and we use the van der Corput sequence ${\bf x}_C$, then $\mathcal{E}_N(f;{\bf x}_C)=\mathcal{O}(\log(N)/N)$.

A drawback of (\ref{Koksma}) is that it provides no error estimate in the case when $f\notin BV_1$. For instance, we were originally interested in finding a general error estimates for $f\in BV_p~~(p>1)$ (see Corollary \ref{pVarCor} below). This led us to our main result (Theorem \ref{MainTeo}), which is a sharpening of (\ref{Koksma}) that is effective also when $f\notin BV_1$. In fact, Theorem \ref{MainTeo} provides an estimate of (\ref{MainError}) for \emph{any} function.
For this, we recall the notion of \emph{modulus of variation}, first introduced in \cite{Lagrange} (see also \cite{Chanturiya1}). For any $N\in\mathbb{N}$, we set
$$
\nu(f;N)=\sup\sum_{j=1}^N|f(I_j)|,
$$
where the supremum is taken over all non-overlapping collections of \emph{at most} $N$ sub-intervals of $[0,1]$. An attractive feature of the modulus of variation is that it is finite for any bounded function. Of course, $f\in BV_1$ if and only if $\nu(f;N)=\mathcal{O}(1)$ as $N\rightarrow\infty$ and the growth of $\nu(f;N)$ then tells us how "badly" a function has unbounded 1-variation. The next result is our main theorem.

%Note that if $f\in BV_1$ and the van der Corput sequence is used, we have 
%\begin{equation}
  %  \label{almostOptimal}
   % \mathcal{E}_N(f;{\bf x}_C)=\mathcal{O}\left(\frac{\log(N)}{N}\right),
%\end{equation}
%by (\ref{vandetCorput}).

%On the other hand, it is not unreasonable to expect some guaranteed rate of convergence of QMC if either
%\begin{enumerate}
    %\item \label{BVList} $f\in BV_p$ for some $p>1$, or
   % \item \label{HolderList} $f$ has additional continuity properties, such as H\"{o}lder continuity.
%\end{enumerate}
%(We give a precise description of the situation (\ref{HolderList}) in Remark \ref{HomegaDef} below.) 

%In particular, this yields a unified approach to establishing convergence rates of $\mathcal{E}_N(f;{\bf x})$ in the cases (\ref{BVList}) and (\ref{HolderList}) above.

%For instance, if $f\in BV_p\setminus BV_1$ for some $p>1$, then $\nu(f;N)\rightarrow\infty$ but not too fast. Indeed, it follows immediately from H\"{o}lder's inequality that $\nu(f;N)=\mathcal{O}(N^{1-1/p})$.

\begin{teo}
\label{MainTeo}
There is an absolute constant $C>0$ such that for any function $f$ and $N\in\mathbb{N}$ there holds
\begin{equation}
    \label{KoksmaNew}
    \mathcal{E}_N(f;{\bf x})\le CD_N^*({\bf x})\nu(f;N).
\end{equation}
\end{teo}
We can take $C=25$ in (\ref{KoksmaNew}). This value is a consequence of our method of proof and not optimal. The main point is that we can replace the total variation in (\ref{Koksma}) with a quantity that is finite for all functions.

On the other hand, we can improve the constant of (\ref{KoksmaNew}) if $f$ is continuous.
\begin{teo}
\label{continuityTeo}
For any continuous function $f$ and $N\in\mathbb{N}$ there holds
\begin{equation}
    \label{KoksmaContinuity}
    \mathcal{E}_N(f;{\bf x})\le D_N^*({\bf x})\nu(f;2N+2).
\end{equation}
\end{teo}
%Observe that (\ref{Koksma}) for continuous $f\in BV_1$ follows directly from (\ref{KoksmaContinuity}).

We shall point out some corollaries of (\ref{KoksmaNew}). First, we have the following Koksma-type inequality for $BV_p$.
\begin{cor}
\label{pVarCor}
Let $f\in BV_p~~(1\le p<\infty)$, there is a constant $C>0$ such that for all $N\in\mathbb{R}$ there holds
\begin{equation}
    \label{pVarEst}
    \mathcal{E}_N(f;{\bf x})\le CN^{1-1/p}D_N^*({\bf x}){\rm Var}_p(f).
\end{equation}
\end{cor}
We suspect that (\ref{pVarEst}) is known, but we have not been able to locate it in the literature. 

We shall also discuss error estimates for functions with some continuity properties. Our result here (Corollary \ref{HolderCor}) is known, see \cite{KN}, but Theorem \ref{MainTeo} allows us to derive it in a very simple way. First, we need to recall the \emph{modulus of continuity} of $f$:
$$
\omega(f;\delta)=\sup_{|x-y|\le\delta}|f(x)-f(y)|\quad(0\le\delta\le1).
$$
Let $\omega:[0,1]\rightarrow[0,\infty)$ be a non-decreasing function with $\omega(0)=0$, strictly concave and differentiable on (0,1). We denote by $H^\omega$ the class of functions such that
$$
|f|_{H^\omega}=\sup_{0<\delta\le1}\frac{\omega(f;\delta)}{\omega(\delta)}<\infty.
$$
In particular, if $\omega(\delta)=\delta^\alpha~~(0<\alpha<1)$, then $H^\omega$ is the space of $\alpha$-H\"{o}lder continuous functions.

\begin{cor}[see e.g. \cite{KN}, p. 146]
\label{HolderCor}
Let $f\in H^\omega$, there is a constant $C>0$ such that for all $N\in\mathbb{N}$ there holds
\begin{equation}
    \label{ErrorHomega}
    \mathcal{E}_N(f;{\bf x})\le C|f|_{H^\omega}N\omega\left(\frac{1}{N}\right)D_N^*({\bf x}).
\end{equation}
\end{cor}
Corollaries \ref{pVarCor} and \ref{HolderCor} follow from simple estimates of the quantity $\nu(f;N)$ (see Proposition \ref{CorollaryProof} below).

It is perhaps not clear that Theorem \ref{MainTeo} really provides any additional information compared to Koksma's inequality (\ref{Koksma}), or the previously known results Corollaries \ref{pVarCor}-\ref{HolderCor}. In the following example, we show that (\ref{KoksmaNew}) can give a better convergence rate of $\mathcal{E}_N$ than (\ref{Koksma}), (\ref{pVarEst}) or (\ref{ErrorHomega}).
\begin{example}
Let $g(x)=x\sin(1/x)$, $x\in(0,1]$ and $g(0)=0$. It is easy to see that $g\in BV_p$ for any $p>1$ but $g\notin BV_1$. Hence, (\ref{Koksma}) does not apply. On the other hand, by Corollary \ref{pVarCor}
$$
\mathcal{E}_N(g;{\bf x}_C)\le C(\alpha)\frac{\log N}{N^\alpha}
$$
for any $\alpha<1$, where ${\bf x}_C$ is the van der Corput sequence. Here, $C(\alpha)\rightarrow\infty$ as $\alpha\rightarrow1-$. (This result can also be obtained from Corollary \ref{HolderCor} since $\omega(g;\delta)=\mathcal{O}(\delta^\alpha)$ for any $\alpha\in(0,1)$.) However, we get a sharper result if we use Theorem \ref{MainTeo} directly. Indeed, it is not difficult to see that $\nu(g;N)=\mathcal{O}(\log(N))$. Whence,
$$
\mathcal{E}_N(g;{\bf x}_C)=\mathcal{O}\left(\frac{(\log N)^2}{N}\right),
$$
where the implied constant is absolute.
\end{example}

%In particular, if $f\in BV_p~~(1\le p<\infty)$ and we take the van der Corput sequence ${\bf x}_C$ in (\ref{pVarEst}), we have
%$$
%\mathcal{E}_N(f;{\bf x}_C)=\mathcal{O}\left(\frac{\log(N)}{N^{1/p}}\right).
%$$

\section{Auxiliary results}
In this section, we collect some auxiliary results that will be useful to us.

For the sake of the reader, we  start by briefly recalling some facts from  discrepancy theory. We follow closely the presentation given in \cite{KN}.

Let $N$ be fixed and consider a finite sequence ${\bf x}=\{x_n:1\le n\le N\}$. When calculating the discrepancy, the order of the elements does not matter. Thus, we can always assume that $x_1\le x_2\le...\le x_N$. Furthermore, we also have
\begin{lem}{\cite[p. 91]{KN}}
    \label{discreteDiscrepancyLemma}
    Assume that $x_1\le x_2\le...\le x_N$, then 
    \begin{eqnarray}
        \label{discreteDiscrepancy}
        D_N^*({\bf x})&=&\max_{1\le n\le N}\left\{\left|x_n-\frac{n-1}{N}\right|,\left|x_n-\frac{n}{N}\right|\right\}\\
        \nonumber
        &=&\frac{1}{2N}+\max_{1\le n\le N}\left|x_n-\frac{2n-1}{2N}\right|
    \end{eqnarray}
\end{lem}
Hence, for any finite sequence ${\bf x}$ of $N$ points, we have
\begin{equation}
    \label{lowerFinite}
    D^*_N({\bf x})\ge\frac{1}{2N},
\end{equation}
with equality if and only if ${\bf x}$ is a permutation of $\{(2n-1)/(2N): 1\le n\le N\}$.

For an infinite sequence, the optimal rate of decay of the discrepancy is $\mathcal{O}(\log(N)/N)$. This rate is attained by the above mentioned van der Corput sequence ${\bf x}_C$. The fact that $\log(N)/N$ is optimal is due to the following result of Schmidt (see e.g. \cite[p. 109]{KN}):
for any infinite sequence ${\bf x}$ there holds
$$
\liminf_{N\rightarrow\infty}\frac{ND_N^*({\bf x})}{\log(N)}>0.
$$

Assume as above that $x_1\le x_2\le...\le x_N$, and set $x_0=0$ and $x_{N+1}=1$.
Then the following identity holds (see \cite[p. 143]{KN})
\begin{equation}
    \label{Zaremba}
    \frac{1}{N}\sum_{n=1}^Nf(x_n)-\int_0^1f(t)dt=\sum_{n=0}^N\int_{x_n}^{x_{n+1}}\left(t-\frac{n}{N}\right)df.
\end{equation}
Koksma's inequality (\ref{Koksma}) follows easily from (\ref{Zaremba}) and (\ref{discreteDiscrepancy}). The author could not derive (\ref{KoksmaNew}) from (\ref{Zaremba}) and thus found the argument presented in the proof of Theorem \ref{MainTeo} below. However, the proof of Theorem \ref{continuityTeo} \emph{does} make essential use of (\ref{Zaremba}) together with the following mean value theorem due to Hobson \cite{Hobson} (see also \cite{Dixon}).
\begin{lem}
\label{HobsonMVT}
Let $\varphi$ be a non-constant function that is monotone on the open interval $(a,b)$ and let $f$ be integrable on $(a,b)$. Then there exists $c\in(a,b)$ such that
\begin{equation}
\nonumber
\int_a^b\varphi(t)f(t)dt=\varphi(a+)\int_a^cf(t)dt+\varphi(b-)\int_c^bf(t)dt,
\end{equation}
where
$$
\varphi(a+)=\lim_{h\rightarrow0, h>0}\varphi(a+h)\quad{\rm and}\quad \varphi(b-)=\lim_{h\rightarrow0, h>0}\varphi(b-h).
$$
\end{lem}
In proving Theorems \ref{MainTeo} and \ref{continuityTeo}, we use some approximation arguments and need to control the variation of the approximants. This is accomplished by the following two simple lemmas. 
\begin{lem}
\label{approximationLemma}
Let $M\in\mathbb{N}$ and let $s_M$ be the continuous first-order spline interpolating $f$ at the knots
$0=x_0<x_1<...<x_M=1$. Then
\begin{equation}
\label{varIneq1}
{\rm Var}_1(s_M)\le\nu(f;M).
\end{equation}
\end{lem}
\begin{proof}
Let $\{x_{n_k}\}$ be the subset of $\{x_n\}$ consisting of points of local extremum of $s_M$. It is easy to see that
$$
{\rm Var}_1(s_M)=\sum_k|s_M(x_{n_{k+1}})-s_M(x_{n_k})|=\sum_k|f(x_{n_{k+1}})-f(x_{n_k})|\le \nu(f;M),
$$
where the last inequality holds since the sum extends over at most $M$ terms.
\end{proof}
\begin{lem}
\label{steklovLemma}
Let $h>0$ be fixed but arbitrary and set
\begin{equation}
    \label{steklovMean}
    f_h(x)=\frac{1}{h}\int_0^hf(x+t)dt
\end{equation}
for $x\in[0,1]$ (assume that $f(x)=f(1)$ for $x\in[1,1+h]$). Then
\begin{equation}
    \nonumber
    \nu(f_h;k)\le\nu(f;k)
\end{equation}
for any $k\in\mathbb{N}$.
\end{lem}
\begin{proof}
Let $\{[a_n,b_n]:1\le n\le k\}$ be an arbitrary collection of $k$ non-overlapping intervals. Then
\begin{eqnarray}
    \nonumber
    \sum_{n=1}^k|f_h(b_n)-f_h(a_n)|&=&\sum_{n=1}^k\left|\frac{1}{h}\int_0^h[f(b_n+t)-f(a_n+t)]dt\right|\\
    \nonumber
    &\le&\frac{1}{h}\int_0^h\sum_{n=1}^k|f(b_n+t)-f(a_n+t)|dt\\
    \nonumber
    &\le&\nu(f;k).
\end{eqnarray}
\end{proof}

\section{Proofs}

%\begin{lem}
 %   \label{largestInterval}
 %   Set
%    $$
 %   \delta_N({\bf x})=\max\left\{|J|: J\cap\{x_1,x_2,...,x_N\}=\emptyset\right\}.
 %   $$
%    Then $\delta_N({\bf x})\le 2D_N^*({\bf x})$.
%\end{lem}
%\begin{proof}
%The \emph{discrepancy} of ${\bf x}$ is defined as 
%$$
%D_N({\bf x})=\sup_{J\subset[0,1]}\left|\frac{\sharp(J\cap\{x_1,x_2,...,x_N\})}{N}-|J|\right|.
%$$
%It is well-known that $D_N({\bf x})\le 2D_N^*({\bf x})$ (see \cite{KN}, p. 91). Further, we observe that for any $J$ with $J\cap\{x_1,x_2,...,x_N\}=\emptyset$, there holds
%$$
%|J|=\left|\frac{\sharp(J\cap\{x_1,x_2,...,x_N\})}{N}-|J|\right|\le D_N({\bf x}),
%$$
%whence $\delta_N({\bf x})\le 2D_N^*({\bf x})$.
%\end{proof}

%We also have the following representation theorem for the integration error in terms of a Riemann-Stieltjes integral.
%\begin{lem}[see e.g. \cite{KN}, p. 143]
 %   \label{Zarembatype}
 %   Let $f\in BV_1$ and $x_1<x_2<...<x_N$. Set further $x_0=0, x_{N+1}=1$, then
  %  \begin{equation}
  %      \frac{1}{N}\sum_{n=1}^N f(x_n)-\int_0^1f(t)dt=\sum_{n=0}^N\int_{x_n}^{x_{n+1}}\left(t-\frac{n}{N}\right)df.
 %   \end{equation}
%\end{lem}

In this section we prove our results. Throughout the section, ${\bf x}$ is an infinite sequence and $N\in\mathbb{N}$ a fixed number.

\begin{proof}[Proof of Theorem \ref{MainTeo}]
Without loss of generality, we may assume that 
$$
x_1<x_2<...<x_N.
$$ 
Set $x_0=0$ and $x_{N+1}=1$ and let $s_N$ be the continuous first-order spline interpolating $f$ at the knots $x_0, x_1,...,x_N, x_{N+1}$.
Since $f(x_n)=s_N(x_n)$ for $0\le n\le N+1$, we have
$$
\frac{1}{N}\sum_{n=1}^N f(x_n)-\int_0^1f(t)dt=\frac{1}{N}\sum_{n=1}^N s_N(x_n)-\int_0^1s_N(t)dt+\int_0^1[s_N(t)-f(t)]dt.
$$
Hence,
$$
\mathcal{E}_N(f;{\bf x})\le\mathcal{E}_N(s_N,{\bf x})+\|f-s_N\|_{L^1(0,1)}.
$$
By (\ref{Koksma}) and Lemma \ref{approximationLemma}, we have
\begin{equation}
    \label{errorEst}
    \mathcal{E}_N(f;{\bf x})\le D_N^*({\bf x})\nu(f;N)+\|f-s_N\|_{L^1(0,1)}.
\end{equation}
We shall estimate $\|f-s_N\|_{L^1(0,1)}$. We have
$$
\|f-s_N\|_{L^1(0,1)}=\sum_{n=0}^{N}\int_{x_n}^{x_{n+1}}|f(t)-s_N(t)|dt.
$$
Observe now that for each $n=0,1,...,N$, there are $y_n\in(x_n,x_{n+1})$ such that
\begin{equation}
    \label{contra}
   \int_{x_n}^{x_{n+1}}|f(t)-s_N(t)|dt\le2(x_{n+1}-x_n)|f(y_n)-s_N(y_n)|. 
\end{equation}
Indeed, assume that (\ref{contra}) is false for some $k$, i.e.,
$$
\frac{1}{(x_{k+1}-x_k)}\int_{x_k}^{x_{k+1}}|f(t)-s_N(t)|dt> 2|f(t)-s_N(t)|
$$
holds for all $t\in(x_k,x_{k+1})$. Integrating the previous inequality over $(x_k,x_{k+1})$ gives
$$
\int_{x_k}^{x_{k+1}}|f(t)-s_N(t)|dt\ge 2\int_{x_k}^{x_{k+1}}|f(t)-s_N(t)|dt,
$$
which is a contradiction since we may assume that $\{f\neq s_N\}$ has positive measure. 
Thus, by (\ref{contra}),
\begin{equation}
    \label{L1estimate}
    \|f-s_N\|_{L^1(0,1)}\le2\delta_N({\bf x})\sum_{n=0}^{N}|f(y_n)-s_N(y_n)|,
\end{equation}
where $\delta_N({\bf x})=\max_{0\le n\le N}(x_{n+1}-x_n)$.
We shall first prove that
\begin{equation}
    \label{largestInterval}
    \delta_N({\bf x})\le 4D_N^*({\bf x}).
\end{equation}
For $1\le n\le N-1$, we have
\begin{eqnarray}
    \nonumber
    x_{n+1}-x_n&\le&\left|x_n-\frac{2n-1}{2N}\right|+\left|x_{n+1}-\frac{2n+1}{2N}\right|+\frac{1}{N}\\
    \nonumber
    &\le& 2D_N^*({\bf x}),
\end{eqnarray}
by (\ref{lowerFinite}) and (\ref{discreteDiscrepancy}).
For $n=0$, set $J_0=[x_0,x_1)$ and observe that
$$
x_1-x_0=|J_0|=\left|\frac{A_N(J_0,{\bf x})}{N}-|J_0|\right|\le D_N({\bf x})\le 2D_N^*({\bf x}),
$$
by (\ref{discComparison}). Finally, for $n=N$, set $J_N=[x_N,x_{N+1})$ and observe that $A_N(J_N,{\bf x})=1$. Thus,
\begin{eqnarray}
    \nonumber
    x_{N+1}-x_N=|J_N|&\le&\left|\frac{A_N(J_N,{\bf x})}{N}-|J_N|\right|+\frac{1}{N}\le D_N({\bf x})+\frac{1}{N}\\
    \nonumber
    &\le& 2D_N^*({\bf x})+\frac{1}{N}\\
    \nonumber
    &\le& 4D_N^*({\bf x}).
\end{eqnarray}
This proves (\ref{largestInterval}). Hence, by (\ref{L1estimate}), we have
$$
\|f-s_N\|_{L^1(0,1)}\le 8D_N^*({\bf x})\sum_{n=0}^N|f(y_n)-s_N(y_n)|.
$$
Furthermore, 
\begin{eqnarray}
    \nonumber
    |f(y_n)-s_N(y_n)|&\le& |f(y_n)-f(x_n)|+|f(x_n)-s_N(y_n)|\\
    \nonumber
    &=&|f(y_n)-f(x_n)|+|s_N(x_n)-s_N(y_n)|
\end{eqnarray}
and it follows that
\begin{eqnarray}
    \nonumber
    \sum_{n=0}^N|f(y_n)-s_N(y_n)|&\le&\sum_{n=0}^N|f(y_n)-f(x_n)|+\sum_{n=0}^N|s_N(x_n)-s_N(y_n)|\\
    \nonumber
    &\le&\nu(f;N+1)+{\rm Var}_1(s_N)\\
    \nonumber
    &\le&\nu(f;N+1)+\nu(f;N)\\
    \nonumber
    &\le& 3\nu(f;N),
\end{eqnarray}
where we used the obvious inequality $\nu(f;N+1)\le 2\nu(f;N)$.
Consequently,
$$
\|f-s_N\|_{L^1(0,1)}\le24 D_N^*({\bf x})\nu(f;N).
$$
and by (\ref{errorEst}) we obtain 
$$
\mathcal{E}_N(f;{\bf x})\le 25D_N^*({\bf x})\nu(f;N).
$$
\end{proof}

\begin{proof}[Proof of Theorem \ref{continuityTeo}]

We will use the identity (\ref{Zaremba}).
Assume first that $f\in C^1(0,1)$  (i.e. has a continuous derivative on $(0,1)$). In this case
$$
\int_{x_n}^{x_{n+1}}\left(t-\frac{n}{N}\right)df=\int_{x_n}^{x_{n+1}}\left(t-\frac{n}{N}\right)f'(t)dt,
$$
see e.g. \cite{WZ}.
Since $t\mapsto(t-n/N)$ is monotone on $(x_n,x_{n+1})$ and $f'$ is integrable, we may use Lemma \ref{HobsonMVT} to get
$$
\int_{x_n}^{x_{n+1}}\left(t-\frac{n}{N}\right)df=(x_n-n/N)\int_{x_n}^{y_n}f'(t)dt+(x_{n+1}-n/N)\int_{y_n}^{x_{n+1}}f'(t)dt
$$
for some $y_n\in(x_n,x_{n+1})$.
Hence, by (\ref{discreteDiscrepancy}) and the fundamental theorem of calculus, we have for $n=0,1,...,N$
\begin{eqnarray}
\nonumber
\left|\int_{x_n}^{x_{n+1}}\left(t-\frac{n}{N}\right)df\right|&\le&\left|x_n-\frac{n}{N}\right|\left|\int_{x_n}^{y_n}f'(t)dt\right|+\left|x_{n+1}-\frac{n}{N}\right|\left|\int_{y_n}^{x_{n+1}}f'(t)dt\right|\\
\nonumber
&\le& D_N^*({\bf x})\left(|f(J'_n)|+|f(J''_n)|\right),
\end{eqnarray}
where $J_n'=[x_n,y_n]$ and $J_n''=[y_n,x_{n+1}]$. It follows that
\begin{eqnarray}
\nonumber
\mathcal{E}_N(f;{\bf x})&\le&D_N^*({\bf x})\sum_{n=0}^N\left(|f(J_n')|+|f(J_n'')|\right)\\
\nonumber
&\le& D_N^*({\bf x})\nu(f;2N+2).
\end{eqnarray}
Thus, (\ref{KoksmaContinuity}) holds for any $C^1$-function. To conclude the argument, we approximate a continuous $f$ with functions in $C^1$ in the following way. Fix $h>0$ and let $f_h$ be given by (\ref{steklovMean}). Since $f$ is continuous, $f_h\in C^1(0,1)$. Hence, for any $h>0$
\begin{equation}
\label{MainTeo2ineq1}
\mathcal{E}_N(f_h;{\bf x})\le D^*_N({\bf x})\nu(f_h;2N+2).
\end{equation}
Further, by Lemma \ref{steklovLemma}, we have for any $k\in\mathbb{N}$
\begin{equation}
\label{MainTeo2ineq2}
\nu(f_h;k)\le\nu(f;k)
\end{equation}
Combining (\ref{MainTeo2ineq1}) and (\ref{MainTeo2ineq2}), we get
\begin{equation}
\label{MainTeo2ineq3}
\mathcal{E}_N(f_h;{\bf x})\le D^*_N({\bf x})\nu(f;2N+2)
\end{equation}
for any $h>0$. It is easy to see that $\|f-f_h\|_\infty\le\omega(f;h)$ (where $\omega(f;h)$ denotes the modulus of continuity of $f$), whence $f_h\rightarrow f$ uniformly since $f$ is continuous. Uniform convergence allows us to interchange the order of limit and integration to obtain
$$
\lim_{h\rightarrow0+}\mathcal{E}_N(f_h;{\bf x})=\mathcal{E}_N(f;{\bf x}).
$$
Hence, letting $h\rightarrow0+$ in (\ref{MainTeo2ineq3}), we obtain (\ref{KoksmaContinuity}).
\end{proof}

The next result proves the corollaries \ref{pVarCor} and \ref{HolderCor}.
\begin{prop}
    \label{CorollaryProof}
    Let $N\in\mathbb{N}$, then we have
    \begin{equation}
        \label{pEstim}
        \nu(f;N)\le N^{1-1/p}{\rm Var}_p(f)\quad(1\le p<\infty),
    \end{equation}
    and
    \begin{equation}
        \label{holderEstim}
        \nu(f;N)\le |f|_{H^\omega}N\omega\left(\frac{1}{N}\right).
    \end{equation}
\end{prop}
\begin{proof}
The inequality (\ref{pEstim}) follows immediately from H\"{o}lder's inequality.
For (\ref{holderEstim}), take $N$ intervals $\{I_j\}_{j=1}^N$, then clearly
$$
\sum_{j=1}^N|f(I_j)|\le |f|_{H^\omega}\sum_{j=1}^N\omega(|I_j|).
$$
Define
\begin{equation}
    \label{maximum}
    M_\omega(N)=\max\left(\sum_{j=1}^N\omega(t_j)\right) \quad{\rm subject\,\, to}\quad \sum_{j=1}^N t_j=1,\quad t_j>0.
\end{equation}
Since $\sum_{j=1}^N|I_j|\le 1$ and $\omega$ is non-decreasing, we clearly have
$$
\nu(f;N)\le |f|_{H^\omega}M_\omega(N).
$$
To calculate $M_\omega(N)$, we use Lagrange multipliers. The critical point $(t_1,t_2,...,t_N,\lambda)$ of the Lagrangian function solves
$$
\omega'(t_j)-\lambda=0\quad(1\le j\le N)\quad{\rm and}\quad\sum_{j=1}^Nt_j-1=0.
$$
By the strict concavity of $\omega$, the above system has the unique solution $t_1=t_2=...=t_N$. Hence, the maximum (\ref{maximum}) is 
$$
M_\omega(N)=N\omega\left(\frac{1}{N}\right).
$$
\end{proof}

\end{document}